\documentclass[reqno]{amsart}
\usepackage{amsmath}
\usepackage{eurosym}
\usepackage{amsfonts}
\usepackage{amssymb}
\usepackage{amsmath}
\usepackage{amsfonts}
\usepackage{graphicx}
\usepackage{subfig}
\usepackage{amssymb}
\usepackage{amsmath}
\usepackage{amsfonts}
\usepackage{graphicx}
\usepackage{subfig}
\usepackage{amsfonts}

\newtheorem{theorem}{Theorem}
\theoremstyle{plain}

\newtheorem{definition}{Definition}
\newtheorem{example}{Example}

\numberwithin{equation}{section}

\begin{document}
\title[]{A geometric perspective on the inextensible flows and energy of
curves in 4-dimensional pseudo-Galilean space}
\author{Fatma Almaz}
\address{department of mathematics, faculty of arts and sciences, batman
university, batman/ t\"{u}rk\.{ı}ye}
\email{fatma.almaz@batman.edu.tr}
\author{Handan \"{O}ztek\.{ı}n}
\address{department of mathematics, faculty of science, firat university,
elaz\i \u{g}/ t\"{u}rk\.{ı}ye}
\email{handanoztekin@gmail.com}
\thanks{This paper is in final form and no version of it will be submitted
for publication elsewhere.}

\begin{abstract}
In this study, inextensible flows of curves in four-dimensional pseudo
Galilean space are expressed, and the necessary and sufficient conditions of
these curve flows are given as partial differential equations. Also, the
directional derivatives are defined in accordance with the Serret-Frenet
frame in $\mathbf{G}_{1}^{4}$, the extended Serret-Frenet relations are
expressed by using Frenet formulas in $\mathbf{G}_{1}^{4}$. Furthermore, the
bending elastic energy functions are expressed for the same particle
according to curve $\alpha (s,t)$.
\end{abstract}

\maketitle


\section{Introduction}

Inextensible curves in pseudo-Galilean 4-space $G_{1}^{4}$ are a
particularly interesting topic in differential geometry. In classical
differential geometry, inextensible curves often arise in the modelling of
elastic rods or strings and can be thought of as curves whose arc length
does not change with time. Given the unique metric structure of $G_{1}^{4}$,
the application of this concept presents some differences in $G_{1}^{4}$.

The study of inextensible curves in $G_{1}^{4}$ contributes to a better
understanding of the geometric structure of space. Furthermore, such curves
may have potential applications in modelling certain systems in classical
mechanics or certain physical models. The differential equations of
inextensible curves and their solutions are important for understanding how
motions in $G_{1}^{4}$ behave under constraints, the inextensible of a curve
means that its arc length is insensitive to changes in parameters or changes
during a curve deformation. Also, the inextensible of a curve is related to
the norm of the curve's tangent vector being constant.

The relationship between inextensible curves and energy in pseudo-Galilean
4-space $G_{1}^{4}$ is an important topic in modelling physical systems and
studying differential geometric structures. As we have seen, the concept of
inextensible relates to the curve being "inextensible," that is, the norm of
the tangent vector being constant. Energy, on the other hand, is often
associated with the state of a system and can be used to derive the
equations of motion. Also, in classical mechanics, the equations of motion
are derived using a system's Lagrangian or Hamiltonian. Energy is often
closely related to the Hamiltonian. Geometric constraints, such as
inextensible, limit the possible motions of the system, and these
constraints are reflected in the energy expressions.

The Lagrangian or Hamiltonian of a physical system modelling a curve in $%
G_{1}^{4}$ typically contains terms involving the square of the norm of the
tangent vector. For example, the kinetic energy of a simple particle can be
related to the norm of its velocity vector according to the metric$\ $in $%
G_{1}^{4}$, when considering a curve $\Omega $ in $G_{1}^{4}$, we may want
to define a concept such as kinetic energy of the curve. In $G_{1}^{4}$, the
tangent vector $\Omega ^{\prime }=T$ represents the velocity, and the square
of the norm of the tangent vector $\left\Vert T\right\Vert ^{2}$ is
analogous to the square of the velocity. However, due to the indefinite
nature of the metric in $G_{1}^{4}$, $\left\Vert T\right\Vert ^{2}$ can be
positive, negative, or zero. If the curve is inextensible, $\left\Vert
T\right\Vert ^{2}=A=$constant, this creates a constraint on the energy
expression. This constant can be positive ($A>0$, space-like), negative ($%
A<0 $, time-like), or zero ($A=0$, null). That is;

1) If the curve is space-like inextensible, the kinetic energy term is
positive and constant. This may evoke a similarity to constant-speed motion
in classical mechanics, but the metric structure of $G_{1}^{4}$ complicates
this analogy.

2) If the curve is timelike inextensible, the kinetic energy term is
negative and constant. This situation may represent a situation that has no
direct analogue in classical physics but is related to timelike motion in
the theory of relativity.

3) If the curve is null inextensible, the kinetic energy term is zero. This
is analogous to the situation of massless particles moving at the speed of
light and plays a special role in the expression for energy.

In \cite{2}, energy and electromagnetic field vectors are investigated in
the cone 3-space. In \cite{13,14}, the energy and pseudo angle of Frenet
vector fields for a particle are characterized by the authors in Minkowski
4-space and $R^{n}.$ In \cite{16}, inextensible flows of planar curves were
elaborated in detail, and some examples of the latter were given. In \cite%
{17}, the inextensible flows of curves and developable surfaces are
investigated and given necessary and sufficient conditions for an
inextensible curve flow in $R^{3}$. In \cite{19}, the equivalence between
the motion of the simple pendulum and fundamental differential equation of
elastic were investigated. In \cite{22}, inextensible flows of curves are
given by the authors in pseudo-Galilean space and 4-dimensional Galilean
space. In \cite{8,13}, the effects of geometric phase rotation are expressed
topological features of traditional Maxwellian theory and attain the
comprehensive consequence for arbitrary fiber trajectory are comprehended.
In \cite{15}, the bending elastic energy function for the same particle are
determined by the authors in De-Sitter 3-space. In \cite{14}, the pioneering
connection between the solutions of the cubic non-linear Schröndinger
equation are investigated by the authors. The unit vector field's energy on
a Riemann manifold is described with Sasaki metric, \cite{24}. In \cite{11},
the similar studies are made about volume of a unit vector field defined as
the volume of the submanifold in the unit tangent bundle. In \cite{1}, the
magnetic flow associated with the Killing magnetic field on lightlike cone
is examined and different magnetic curves are found in the 2D lightlike cone
using the Killing magnetic field of these curves. In \cite{4,10}, the
authors express energy and volume of vector fields. The authors investigated
the change of the Willmore energy of curves in 3-dimensional Lorentzian
space, \cite{18}.

In this study, we present the main results on inextensible curve flows
aspects of classical differential geometry topics to the 4-dimensional
pseudo-Galilean Space. We also derive the corresponding differential
equations for inextensible flows. Moreover, we express the energy on the
inextensible curves corresponding to a particular particle and the pseudo
angles of the Frenet-Serret vector fields in $G_{1}^{4}$ along with $s-$
lines(or $t-$lines) coordinate in $G_{1}^{4}$.

\section{Preliminaries}

$G_{1}^{4}$ is a special type of geometry that has a metric structure
different from four-dimensional Euclidean space. The study of curves in this
space plays an important role in differential geometry and related physical
theories. It provides a foundation for further studies of curves in $%
G_{1}^{4}$. The unique metric structure of this space gives rise to
properties in the behaviour of curves that differ from those in Euclidean
space.

In affine coordinates, the pseudo-Galilean scalar product between two points $%
P_{i}=(p_{i1},p_{i2},p_{i3},p_{i4}),$ $i=1,2$ is defined by 
\begin{equation*}
g(P_{1},P_{2})=\{%
\begin{array}{c}
\left\vert p_{21}-p_{11}\right\vert , \\ 
\sqrt{\left\vert
-(p_{22}-p_{12})^{2}+(p_{23}-p_{13})^{2}+(p_{24}-p_{14})^{2}\right\vert },%
\end{array}%
\begin{array}{c}
\text{if }p_{21}\neq p_{11}, \\ 
\text{if \ }p_{21}=p_{11}.%
\end{array}%
\end{equation*}

We define the Pseudo-Galilean cross product for the vectors $%
\overrightarrow{u}$ $=(u_{1},u_{2},u_{3},u_{4}),$ $\overrightarrow{v}%
=(v_{1},v_{2},v_{3},v_{4})$ and $\overrightarrow{w}$ $%
=(w_{1},w_{2},w_{3},w_{4})$ in $G_{4}^{1}$ as follows:
\begin{equation*}
\overrightarrow{u}\wedge \overrightarrow{v}\wedge \overrightarrow{w}%
=\left\vert 
\begin{array}{cccc}
0 & -e_{2} & e_{3} & e_{4} \\ 
u_{1} & u_{2} & u_{3} & u_{4} \\ 
v_{1} & v_{2} & v_{3} & v_{4} \\ 
w_{1} & w_{2} & w_{3} & w_{4}%
\end{array}%
\right\vert ,\text{ if }u_{1}\neq 0\text{ or }v_{1}\neq 0\text{ or }%
w_{1}\neq 0,
\end{equation*}%
\begin{equation*}
\overrightarrow{u}\wedge \overrightarrow{v}\wedge \overrightarrow{w}%
=\left\vert 
\begin{array}{cccc}
-e_{1} & e_{2} & e_{3} & e_{4} \\ 
u_{1} & u_{2} & u_{3} & u_{4} \\ 
v_{1} & v_{2} & v_{3} & v_{4} \\ 
w_{1} & w_{2} & w_{3} & w_{4}%
\end{array}%
\right\vert ,\text{ if }u_{1}=\text{ }v_{1}=w_{1}=0.
\end{equation*}

A vector $X(x,y,z,w)$ is called to be non-isotropic, if $x\neq 0$. All unit
non-isotropic vectors are of the form $(1,y,z,w)$. For isotropic vectors, $%
x=0$ holds.

A non-lightlike vector is a unit vector if $-y^{2}+z^{2}+w^{2}=\pm 1.$

In the pseudo-Galilean 4-space, there are isotropic vectors $%
X(x,y,z,w)$ and four types of isotropic vectors: spacelike ($x=0$, $%
-y^{2}+z^{2}+w^{2}>0$), timelike ($x=0$, $-y^{2}+z^{2}+w^{2}<0$) and two
types of lightlike vectors ($x=0$, $y=\sqrt{z^{2}+w^{2}}$ ). The scalar
product of two vectors $\overrightarrow{U}=(u_{1},u_{2},u_{3},u_{4})$ and $%
\overrightarrow{V}=(v_{1},v_{2},v_{3},v_{4})$ in $G_{4}^{1}$ is defined by

\begin{equation*}
\left\langle \overrightarrow{U},\overrightarrow{V}\right\rangle
_{G_{4}^{1}}=\{%
\begin{array}{c}
u_{1}v_{1},\text{ \ \ \ \ \ \ \ \ \ \ \ \ \ \ \ \ \ \ \ \ \ \ \ \ \ \ if }%
u_{1}\neq 0\text{ or }v_{1}\neq 0, \\ 
-u_{2}v_{2}+u_{3}v_{3}+u_{4}v_{4},\text{ \ if }u_{1}=0\text{ and }v_{1}=0.%
\end{array}%
\end{equation*}

The norm of vector $\overrightarrow{U}=(u_{1},u_{2},u_{3},u_{4})$ is defined
by 
\begin{equation*}
\left\Vert \overrightarrow{U}\right\Vert _{G_{4}^{1}}=\sqrt{\left\vert
\left\langle \overrightarrow{U},\overrightarrow{U}\right\rangle \right\vert
_{G_{4}^{1}}}.
\end{equation*}

A curve $\alpha :I\subset 
\mathbb{R}
\rightarrow G_{4}^{1},$ $\alpha (t)=(x(t),y(t),z(t),w(t))$ is called an
admissible curve if $x^{\prime }(t)\neq 0.$

Let $\alpha :I\subset 
\mathbb{R}
\rightarrow G_{4}^{1},$ $\alpha (s)=(s,y(s),z(s),w(s))$ be an admissible
curve parametrized by arclength $s$ in $G_{4}^{1}.$ Here, we denote
differentiation with respect to $s$ by a dash. The first vector of the
Frenet-Serret frame, that is the tangent vector of $\alpha $ is defined by%
\begin{equation*}
T=\alpha ^{\prime }(s)=(1,y^{\prime }(s),z^{\prime }(s),w^{\prime }(s)).
\end{equation*}%
\ 

Since $T$ is a unit non-isotropic vector, so we can express%
\begin{equation}
\left\langle T,T\right\rangle _{G_{4}^{1}}=1.  \tag{2.1}
\end{equation}

Differentiating the formula (2.1) with respect to $s$, we have%
\begin{equation*}
\left\langle T^{\prime },T\right\rangle _{G_{4}^{1}}=0.
\end{equation*}

The vector function $T^{\prime }$ gives us the rotation measurement of the
curve $\alpha .$ The real valued function%
\begin{equation*}
\kappa (s)=\left\Vert T^{\prime }(s)\right\Vert =\sqrt{\left\vert
-(y^{\prime \prime }(s))^{2}+(z^{\prime \prime }(s))^{2}+(w^{\prime \prime
}(s))^{2}\right\vert }
\end{equation*}%
is called the first curvature of the curve $\alpha .$ We assume that, $%
\kappa (s)\neq 0,$ for all $s\in I.$ Similar to space $G_{3}^{1},$ we define
the principal normal vector%
\begin{equation*}
N(s)=\frac{T^{\prime }(s)}{\kappa (s)}
\end{equation*}
or another words 
\begin{equation}
N(s)=\frac{1}{\kappa (s)}(0,y^{\prime \prime }(s),z^{\prime \prime
}(s),w^{\prime \prime }(s)).  \tag{2.2}
\end{equation}

By the aid of the differentiation of the principal normal vector (2.2), we
define the second curvature function as%
\begin{equation}
\tau (s)=\left\Vert N^{\prime }(s)\right\Vert _{G_{4}^{1}}.  \tag{2.3}
\end{equation}

This real valued function is called torsion of the curve $\alpha .$ The
third vector field, namely binormal vector field of the curve $\alpha $ is
defined by%
\begin{equation}
B_{1}(s)=\frac{1}{\tau (s)}(0,(\frac{y^{\prime \prime }(s)}{\kappa (s)}%
)^{\prime },(\frac{z^{\prime \prime }(s)}{\kappa (s)})^{\prime },(\frac{%
w^{\prime \prime }(s)}{\kappa (s)})^{\prime })  \tag{2.4}
\end{equation}

Thus the vector $B_{1}(s)$ is both perpendicular to $T$ and $N.$ The fourth
unit vector is defined by%
\begin{equation}
B_{2}(s)=\mu T(s)\wedge N(s)\wedge B_{1}(s).  \tag{2.5}
\end{equation}

The coefficient $\mu $ is taken $\pm 1$ to make $+1$ determinant of the
matrix $\left[ T,N,B_{1},B_{2}\right].$ We define the third curvature of the
curve $\alpha $ by the pseudo-Galilean inner product%
\begin{equation}
\sigma =\left\langle B_{1}^{\prime },B_{2}\right\rangle _{G_{4}^{1}}. 
\tag{2.6}
\end{equation}

Here, as well known, the set $\left\{ T,N,B_{1},B_{2},\kappa ,\tau ,\sigma
\right\} $ is called the Frenet-Serret apparatus of the curve $\alpha .$ We
know that the vectors $\left\{ T,N,B_{1},B_{2}\right\} $ are mutually
orthogonal vectors satisfying%
\begin{equation}
\left\langle T,T\right\rangle _{G_{4}^{1}}=1,\left\langle N,N\right\rangle
_{G_{4}^{1}}=\varepsilon _{1},\left\langle B_{1},B_{1}\right\rangle
_{G_{4}^{1}}=\varepsilon _{2},\left\langle B_{2},B_{2}\right\rangle
_{G_{4}^{1}}=\varepsilon _{3},  \tag{2.7}
\end{equation}%
\begin{equation*}
\left\langle T,N\right\rangle _{G_{4}^{1}}=\left\langle T,B_{1}\right\rangle
_{G_{4}^{1}}=\left\langle T,B_{2}\right\rangle _{G_{4}^{1}}=\left\langle
N,B_{1}\right\rangle _{G_{4}^{1}}=\left\langle N,B_{2}\right\rangle
_{G_{4}^{1}}=\left\langle B_{2},B_{2}\right\rangle _{G_{4}^{1}}=0
\end{equation*}%
where 
\begin{equation}
\varepsilon _{3}=\{%
\begin{array}{c}
+1,\text{ if }\varepsilon _{1}=-1\text{ or }\varepsilon _{2}=-1 \\ 
-1,\text{ if }\varepsilon _{1}=1\text{ and }\varepsilon _{2}=1.%
\end{array}
\tag{2.8}
\end{equation}

Consequently, the Frenet-Serret equations for an admissible curve $\alpha
(s)$ are given as

\begin{equation}
\frac{\partial }{\partial s}\left[ 
\begin{array}{c}
T \\ 
N \\ 
B_{1} \\ 
B_{2}%
\end{array}%
\right] =\left[ 
\begin{array}{cccc}
0 & \varepsilon _{1}\kappa & 0 & 0 \\ 
0 & 0 & \varepsilon _{2}\tau & 0 \\ 
0 & -\varepsilon _{2}\tau & 0 & \varepsilon _{3}\sigma \\ 
0 & 0 & -\varepsilon _{2}\sigma & 0%
\end{array}%
\right] \left[ 
\begin{array}{c}
T \\ 
N \\ 
B_{1} \\ 
B_{2}%
\end{array}%
\right] ,  \tag{2.9}
\end{equation}%
\cite{7,20}.

\begin{definition}
Let$(M,\varrho )$ and $\left( N,h\right) $ be two Riemannian manifolds and $%
f:(M,\varrho )\rightarrow \left( N,h\right) $ be a differentiable map
between these two manifolds so that the energy functional is defined by 
\begin{equation}
energy(f)=\frac{1}{2}\int_{M}\underset{a=1}{\overset{n}{\sum }}%
h(df(e_{a}),df(e_{a}))v,  \tag{2.10}
\end{equation}%
where $\left\{ e_{a}\right\} $ is a local basis of the tangent space and $v$
is the canonical volume form in $M$ \cite{3,24}.
\end{definition}

\begin{definition}
Let $Q:T(T^{1}M)\rightarrow T^{1}M$ be the connection map. Then, the
following conditions satisfy

i) $\omega oQ=\omega od\omega$\ and $\omega oQ=\omega o\varpi$ where $%
\varpi:T(T^{1}M)\rightarrow T^{1}M$ is the tangent bundle projection;

ii) for $\varrho \in T_{x}M$ and a section $\xi :M\rightarrow T^{1}M$; we
have 
\begin{equation}
Q(d\xi (\varrho ))=D_{\varrho }\xi ,  \tag{2.11}
\end{equation}%
where $D$ is the Levi-Civita covariant derivative \cite{3,24}.
\end{definition}

\begin{definition}
For $\varsigma _{1},\varsigma _{2}\in T_{\xi }\left( T^{1}M\right) $,
Riemannian metric on $TM$ is defined as 
\begin{equation}
\varrho _{S}(\varsigma _{1},\varsigma _{2})=\varrho (d\omega \left(
\varsigma _{1}\right) ,d\omega \left( \varsigma _{2}\right) )+\varrho
(Q\left( \varsigma _{1}\right) ,Q\left( \varsigma _{2}\right) ).  \tag{2.12}
\end{equation}

Here, as known $\varrho _{S}$ is called the Sasaki metric that also makes
the projection $\omega :T^{1}M\rightarrow M$ a Riemannian submersion \cite%
{3,24}.
\end{definition}

\section{Some characterizations of inextensible flows curves in $G_{1}^{4}$}

In this section, we describe the directional derivatives in accordance with
the Frenet frame $\left\{ T,N,B_{1},B_{2}\right\} $ in $G_{1}^{4}$. Also, we
express the extended Serret-Frenet relations using cone Frenet formulas and
we investigate the bending energy formula for tangent vector of $s$-lines( $%
t $-lines) of elastic curve written by extended Serret-Frenet relations
along the curve $\Omega $ in $G_{1}^{4}$.

We consider an arbitrary admissible curve $\Omega $ with the curvatures $%
\kappa ,\tau ,\sigma \neq 0$ as a curve whose position vector satisfies the
parametric equation given as 
\begin{equation*}
\Omega (s)=\left( s+c^{\ast \ast }\right) T(s)-\left( \frac{\varepsilon _{2}%
}{\tau }\overset{..}{\mu }_{4}(t)+\frac{\sigma }{\tau }\mu _{4}(t)\right)
N(s)-\overset{.}{\mu }_{4}(t)B_{1}(s)+\mu _{4}(s)B_{2}(s),
\end{equation*}%
where 
\begin{equation*}
\mu _{4}(t)=t\int t\left( \int \frac{1}{t^{5}}e^{-\varepsilon _{2}\int \frac{%
\tau }{\overset{.}{\tau }}dt}\left( \frac{\varepsilon _{1}}{\varepsilon
_{2}\varepsilon _{3}}\frac{\kappa \tau }{\sigma }\frac{h}{t}\int \frac{1}{%
t^{5}}e^{\varepsilon _{2}\int \frac{\tau }{\overset{.}{\tau }}%
dt}dt+c_{11}\right) dt\right) dt+c_{12},
\end{equation*}%
$c^{\ast \ast },c_{11},c_{12}\in 
\mathbb{R}
$, $t=\varepsilon _{3}\int \sigma (s)ds.$

Let $\Omega :[0,n]\times \lbrack 0,m]\rightarrow G_{1}^{4}$ be a
one-parameter family of smooth curve in $G_{1}^{4}$, where is the arc length
of initial curve, let $u$ be parametrization variable $0\leq u\leq n$ and
the curve speed $v=\left\Vert \frac{\partial \Omega }{\partial u}\right\Vert 
$, from which it follows that the arclength $s(u)=\int\limits_{0}^{u}vdu$.
Furthermore, $\frac{\partial }{\partial s}=v\frac{\partial }{\partial u}$
and $\partial s=v\partial u$. Then, any flow of $\Omega $ can be represented
as 
\begin{equation}
\frac{\partial \Omega }{\partial t}=f_{1}T+f_{2}N+f_{3}B_{1}+f_{4}B_{2}, 
\tag{3.1}
\end{equation}
for some differentiable functions, $f_{i}$ and $1\leq i\leq 4$. Let the arc
length variation be $s(u,t)=\int\limits_{0}^{v}\frac{\partial v}{\partial t}%
du$, the requirement that the curve not be subject to any elongation or
compression can be expressed by the condition%
\begin{equation}
\frac{\partial s(u,t)}{\partial t}=\int\limits_{0}^{v}\frac{\partial v}{%
\partial t}du=0,\forall u\in \lbrack 0,n].  \tag{3.2}
\end{equation}

Hence, we make the following definition.

\begin{definition}
A curve evolution $\Omega (u,t)$ and its flow $\frac{\partial \Omega }{%
\partial u}$ in $G_{1}^{4}$ are said to be inextensible if $\frac{\partial }{%
\partial t}\left\Vert \frac{\partial \Omega }{\partial u}\right\Vert =0.$
\end{definition}

\begin{theorem}
The flow of $\Omega (u,t)$ in $G_{1}^{4}$ is inextensible if and only if $%
f_{1}$ is constant.
\end{theorem}

\begin{proof}
By using the definition of $\Omega $ and by differentiating $\left\Vert 
\frac{\partial \Omega }{\partial u}\right\Vert =v^{2}=\left\langle \frac{%
\partial \Omega }{\partial u},\frac{\partial \Omega }{\partial u}%
\right\rangle,$ we get

\begin{equation*}
2v\frac{\partial v}{\partial t}=2\left\langle \frac{\partial }{\partial t}(%
\frac{\partial \Omega }{\partial u}),\frac{\partial \Omega }{\partial u}%
\right\rangle
\end{equation*}%
also changing $\frac{\partial }{\partial u}$ and $\frac{\partial }{\partial t%
}$, we can write%
\begin{equation*}
v\frac{\partial v}{\partial t}=\left\langle \frac{\partial }{\partial u}%
(f_{1}\overrightarrow{T}+f_{2}\overrightarrow{N}+f_{3}\overrightarrow{B_{1}}%
+f_{4}\overrightarrow{B_{2}}),\frac{\partial \Omega }{\partial s}\frac{%
\partial s}{\partial u}\right\rangle ;\frac{\partial s}{\partial u}=v,
\end{equation*}%
from Frenet-Serret equations (2.9), we get%
\begin{equation}
\frac{\partial v}{\partial t}=\left\langle 
\begin{array}{c}
\overrightarrow{T},\frac{\partial f_{1}}{\partial u}\overrightarrow{T}%
+(f_{1}v\varepsilon _{1}\kappa +\frac{\partial f_{2}}{\partial s}%
-f_{3}\varepsilon _{1}\tau v)\overrightarrow{N}+ \\ 
(f_{2}v\varepsilon _{2}\tau +\frac{\partial f_{3}}{\partial u}%
-f_{4}\varepsilon _{2}\sigma v)\overrightarrow{B_{1}}+(f_{3}\varepsilon
_{3}\sigma v+\frac{\partial f_{4}}{\partial u}_{4})\overrightarrow{B_{2}},%
\end{array}%
\right\rangle .  \tag{3.4}
\end{equation}

If necessary calculations are made in (3.4), we have 
\begin{equation*}
\frac{\partial v}{\partial t}=\frac{\partial f_{1}}{\partial u}=0\Rightarrow
f_{1}=\text{constant.}
\end{equation*}
\end{proof}

Therefore, from the above theorem, we understand that the evolution of a
curve in $G_{1}^{4}$ is inextensible. From what we understand from
differential geometry, this theorem confirms that the concept of "torsion"
is a local measure of how much a curve bends, while curvature reflects the
bending of the curve. Hence, we give the following theorem regarding
inextensible curves.

\begin{theorem}
Let $\frac{\partial \Omega }{\partial t}=f_{1}\overrightarrow{T}+f_{2}%
\overrightarrow{N}+f_{3}\overrightarrow{B_{1}}+f_{4}\overrightarrow{B_{2}}$
be arc-length parameter curve with the Serret-Frenet vectors $\left\{ 
\overrightarrow{T},\overrightarrow{N},\overrightarrow{B_{1}},\overrightarrow{%
B_{2}}\right\} $ in $G_{1}^{4}$. If the flow of the curve $\Omega $ is
inextensible, then extended Serret-Frenet relations with respect to $t$ are
given as 
\begin{equation}
\frac{\partial }{\partial t}\left[ 
\begin{array}{c}
\overrightarrow{T} \\ 
\overrightarrow{N} \\ 
\overrightarrow{B_{1}} \\ 
\overrightarrow{B_{2}}%
\end{array}%
\right] =\left[ 
\begin{array}{cccc}
0 & \xi _{1} & \xi _{2} & \xi _{3} \\ 
-\varepsilon _{1}\xi _{1} & 0 & \Gamma _{1}\varepsilon _{2} & \Gamma
_{2}\varepsilon _{3} \\ 
-\varepsilon _{2}\xi _{2} & -\Gamma _{1}\varepsilon _{1} & 0 & \Gamma
_{3}\varepsilon _{3} \\ 
-\varepsilon _{3}\xi _{3} & -\Gamma _{2}\varepsilon _{1} & -\Gamma
_{3}\varepsilon _{2} & 0%
\end{array}%
\right] \left[ 
\begin{array}{c}
\overrightarrow{T} \\ 
\overrightarrow{N} \\ 
\overrightarrow{B_{1}} \\ 
\overrightarrow{B_{2}}%
\end{array}%
\right] ,  \tag{3.5}
\end{equation}%
where 
\begin{equation*}
\Gamma _{1}=-\left\langle \frac{\partial B_{1}}{\partial t},N\right\rangle
_{G_{1}^{4}},\Gamma _{2}=\left\langle \frac{\partial B_{2}}{\partial t}%
,N\right\rangle _{G_{1}^{4}},\Gamma _{3}=\left\langle \frac{\partial B_{1}}{%
\partial t},B_{2}\right\rangle _{G_{1}^{4}}
\end{equation*}%
\begin{equation*}
\xi _{1}=\varepsilon _{1}\kappa f_{1}+f_{2}^{\prime }-\varepsilon _{1}\tau
f_{3},\xi _{2}=\varepsilon _{2}\tau f_{2}+f_{3}^{\prime }-\varepsilon
_{2}\sigma f_{4},\xi _{3}=f_{4}^{\prime }+\varepsilon _{3}\sigma f_{3}.
\end{equation*}
\end{theorem}

\begin{proof}
Let $\frac{\partial \Omega }{\partial t}=f_{1}\overrightarrow{T}+f_{2}%
\overrightarrow{N}+f_{3}\overrightarrow{B_{1}}+f_{4}\overrightarrow{B_{2}}$
be arc-length parameter curve with the Serret-Frenet vectors $\{%
\overrightarrow{T},\overrightarrow{N},\overrightarrow{B_{1}},\overrightarrow{%
B_{2}}\}$ in $G_{1}^{4}$ where $f_{i}$, $i\in \left\{ 1,2,3,4\right\} $ are
the velocities in direction of the tangent vector $T$ the normal vector $N$,
the first binormal vector $B_{1}$ and the second binormal vector $B_{2}$,
respectively. Furthermore, let the flow of the curve $\Omega $ be
inextensible, then from previous theorem and (2.9), we can write%
\begin{equation*}
\frac{\partial T}{\partial t}=\frac{\partial }{\partial t}(\frac{\partial
\Omega }{\partial s})=\frac{\partial }{\partial s}(\frac{\partial \Omega }{%
\partial t})
\end{equation*}%
\begin{equation}
\frac{\partial T}{\partial t}=(\varepsilon _{1}\kappa f_{1}+f_{2}^{\prime
}-\varepsilon _{1}\tau f_{3})\overrightarrow{N}+(\varepsilon _{2}\tau
f_{2}+f_{3}^{\prime }-\varepsilon _{2}\sigma f_{4})\overrightarrow{B_{1}}%
+(f_{4}^{\prime }+\varepsilon _{3}\sigma f_{3})\overrightarrow{B_{2}}. 
\tag{3.6}
\end{equation}

Differentiating with respect to $t$ and by using (2.9). Then, we have%
\begin{eqnarray*}
\frac{\partial }{\partial t}\left\langle T,N\right\rangle _{G_{1}^{4}}
&=&0\Rightarrow \left\langle \frac{\partial N}{\partial t},T\right\rangle
_{G_{1}^{4}}=-\kappa f_{1}-\varepsilon _{1}f_{2}^{\prime }+\tau f_{3}; \\
\frac{\partial }{\partial t}\left\langle T,B_{1}\right\rangle _{G_{1}^{4}}
&=&0\Rightarrow \left\langle \frac{\partial B_{1}}{\partial t}%
,T\right\rangle _{G_{1}^{4}}=-\tau f_{2}-\varepsilon _{2}f_{3}^{\prime
}+\sigma f_{4}; \\
\frac{\partial }{\partial t}\left\langle T,B_{2}\right\rangle _{G_{1}^{4}}
&=&0\Rightarrow \left\langle \frac{\partial B_{2}}{\partial t}%
,T\right\rangle _{G_{1}^{4}}=-\varepsilon _{3}f_{4}^{\prime }-\sigma f_{3};
\\
\frac{\partial }{\partial t}\left\langle N,B_{1}\right\rangle _{G_{1}^{4}}
&=&0\Rightarrow \left\langle \frac{\partial B_{1}}{\partial t}%
,N\right\rangle _{G_{1}^{4}}=-\Gamma _{1}; \\
\frac{\partial }{\partial t}\left\langle N,B_{2}\right\rangle _{G_{1}^{4}}
&=&0\Rightarrow \left\langle \frac{\partial B_{2}}{\partial t}%
,N\right\rangle _{G_{1}^{4}}=-\Gamma _{2}; \\
\frac{\partial }{\partial t}\left\langle B_{2},B_{1}\right\rangle
_{G_{1}^{4}} &=&0\Rightarrow \left\langle \frac{\partial B_{1}}{\partial t}%
,B_{2}\right\rangle _{G_{1}^{4}}=-\Gamma _{3}.
\end{eqnarray*}

By using the above equations respectively, we get 
\begin{eqnarray*}
\frac{\partial N}{\partial t} &=&(-\kappa f_{1}-\varepsilon
_{1}f_{2}^{\prime }+\tau f_{3})\overrightarrow{T}+\Gamma _{1}\varepsilon _{2}%
\overrightarrow{B_{1}}+\Gamma _{2}\varepsilon _{3}\overrightarrow{B_{2}} \\
\frac{\partial B_{1}}{\partial t} &=&(-\tau f_{2}-\varepsilon
_{2}f_{3}^{\prime }+\sigma f_{4})\overrightarrow{T}-\Gamma _{1}\varepsilon
_{1}\overrightarrow{N}+\Gamma _{3}\varepsilon _{3}\overrightarrow{B_{2}} \\
\frac{\partial B_{2}}{\partial t} &=&(-\varepsilon _{3}f_{4}^{\prime
}+\sigma f_{3})\overrightarrow{T}-\Gamma _{2}\varepsilon _{1}\overrightarrow{%
N}-\Gamma _{3}\varepsilon _{2}\overrightarrow{B_{1}}.
\end{eqnarray*}
\end{proof}

Now, we express the conditions on the curvature and the torsion for the
curve flow $\Omega $\ to be inelastic.

\begin{theorem}
Let $\frac{\partial \Omega }{\partial t}=f_{1}\overrightarrow{T}+f_{2}%
\overrightarrow{N}+f_{3}\overrightarrow{B_{1}}+f_{4}\overrightarrow{B_{2}}$
be arc-length parameter curve in $G_{1}^{4}$. If the flow of the curve $%
\Omega $ is inextensible then following partial differential equations hold%
\begin{eqnarray*}
\Gamma _{1} &=&\frac{1}{\kappa }\left( \varepsilon _{1}\varepsilon _{2}\xi
_{2}^{\prime }+\varepsilon _{1}\xi _{1}\tau -\varepsilon _{1}\xi _{3}\sigma
\right) ;\Gamma _{2}=\frac{1}{\kappa }\left( \varepsilon _{1}\varepsilon
_{3}\xi _{3}^{\prime }+\varepsilon _{1}\xi _{2}\sigma \right) \\
\kappa (t) &=&\int \left( \varepsilon _{1}\xi _{1}^{\prime }-\xi _{2}\tau
\right) dt;\tau (t)=\frac{1}{f_{3}}\left( \kappa (t)f_{1}+\varepsilon
_{1}f_{2}^{\prime }\right)
\end{eqnarray*}%
and for the equation $f_{4}(t)=\int \left( \varepsilon _{1}\varepsilon
_{2}\int \kappa (t)\Gamma _{2}dt+c_{2}\right) dt,$ the following equation
holds%
\begin{equation*}
\sigma (t)=\frac{e^{\int \frac{(\varepsilon _{2}\tau f_{2}+2f_{3}^{\prime })%
}{f_{3}}dt}}{-\varepsilon _{1}\int \frac{f_{4}}{f_{3}}e^{-\int \frac{%
(\varepsilon _{2}\tau f_{2}+2f_{3}^{\prime })}{f_{3}}dt}dt+c_{3}},
\end{equation*}%
where $\xi _{1}=\varepsilon _{1}\kappa f_{1}+f_{2}^{\prime }-\varepsilon
_{1}\tau f_{3},$ $\xi _{2}=\varepsilon _{2}\tau f_{2}+f_{3}^{\prime
}-\varepsilon _{2}\sigma f_{4},$ $\xi _{3}=f_{4}^{\prime }+\varepsilon
_{3}\sigma f_{3}.$

\begin{proof}
From the equality $\frac{\partial }{\partial s}\left( \frac{\partial T}{%
\partial t}\right) =\frac{\partial }{\partial t}\left( \frac{\partial T}{%
\partial s}\right) $ by making necessary arrangements, we get 
\begin{equation*}
\frac{\partial }{\partial s}\left( \frac{\partial T}{\partial t}\right) =%
\frac{\partial }{\partial s}\left( \xi _{1}\overrightarrow{N}+\xi _{2}%
\overrightarrow{B_{1}}+\xi _{3}\overrightarrow{B_{2}}\right)
\end{equation*}%
\begin{equation}
=\left( \xi _{1}^{\prime }-\varepsilon _{1}\xi _{2}\tau \right) 
\overrightarrow{N}+\left( \xi _{2}^{\prime }+\varepsilon _{2}\xi _{1}\tau
-\varepsilon _{2}\xi _{3}\sigma \right) \overrightarrow{B_{1}}+\left( \xi
_{3}^{\prime }+\varepsilon _{3}\xi _{2}\sigma \right) \overrightarrow{B_{2}}
\tag{3.7}
\end{equation}%
where $\xi _{1}=\varepsilon _{1}\kappa f_{1}+f_{2}^{\prime }-\varepsilon
_{1}\tau f_{3},$ $\xi _{2}=\varepsilon _{2}\tau f_{2}+f_{3}^{\prime
}-\varepsilon _{2}\sigma f_{4},$ $\xi _{3}=f_{4}^{\prime }+\varepsilon
_{3}\sigma f_{3}$ and by using the equation $\frac{\partial \overrightarrow{N%
}}{\partial t}$, then we get 
\begin{equation*}
\frac{\partial }{\partial t}\left( \frac{\partial T}{\partial s}\right) =%
\frac{\partial }{\partial t}\left( \kappa \xi _{1}\overrightarrow{N}\right)
=\varepsilon _{1}\frac{\partial \kappa }{\partial t}\overrightarrow{N}%
+\varepsilon _{1}\kappa \frac{\partial \overrightarrow{N}}{\partial t}
\end{equation*}%
\begin{equation}
=\varepsilon _{1}\frac{\partial \kappa }{\partial t}\overrightarrow{N}%
+\varepsilon _{1}\kappa \left( -\kappa f_{1}-\xi _{1}f_{2}^{\prime }+\tau
f_{3}\right) \overrightarrow{T}+\varepsilon _{1}\varepsilon _{2}\kappa
\Gamma _{1}\overrightarrow{B_{1}}+\varepsilon _{1}\kappa \varepsilon
_{3}\Gamma _{2}\overrightarrow{B_{2}}.  \tag{3.8}
\end{equation}

Applying the compatibility condition for equations (3.7) and (3.8), then we
get%
\begin{equation*}
\varepsilon _{1}\frac{\partial \kappa }{\partial t}=\xi _{1}^{\prime
}-\varepsilon _{1}\xi _{2}\tau
\end{equation*}%
\begin{equation*}
\varepsilon _{1}\kappa \left( -\kappa f_{1}-\varepsilon _{1}f_{2}^{\prime
}+\tau f_{3}\right) =0
\end{equation*}%
\begin{equation*}
\xi _{2}^{\prime }+\varepsilon _{2}\xi _{1}\tau -\varepsilon _{2}\xi
_{3}\sigma =\varepsilon _{1}\varepsilon _{2}\kappa \Gamma _{1}
\end{equation*}%
\begin{equation*}
\xi _{3}^{\prime }+\varepsilon _{3}\xi _{2}\sigma =\varepsilon
_{1}\varepsilon _{3}\kappa \Gamma _{2},
\end{equation*}%
from the first equation we get%
\begin{equation}
\kappa (t)=\int \left( \varepsilon _{1}\xi _{1}^{\prime }-\xi _{2}\tau
\right) dt,  \tag{3.9}
\end{equation}%
from the second equation we get%
\begin{equation}
\tau (t)=\frac{1}{f_{3}}\left( \kappa (t)f_{1}+\varepsilon _{1}f_{2}^{\prime
}\right) ,  \tag{3.10}
\end{equation}%
and from the third equation we get 
\begin{equation}
\Gamma _{1}=\frac{1}{\kappa }\left( 
\begin{array}{c}
\varepsilon _{1}\tau ^{\prime }f_{2}+\kappa \tau f_{1}+\varepsilon
_{1}f_{2}^{\prime }\tau (1+\varepsilon _{2}) \\ 
-\varepsilon _{1}f_{3}(\tau ^{2}\varepsilon _{2}+\sigma ^{2}\varepsilon
_{3})-\varepsilon _{1}\sigma f_{4}^{\prime }(\varepsilon _{2}\varepsilon
_{3}+1)+\varepsilon _{2}\varepsilon _{1}\left( f_{3}^{\prime \prime
}-\varepsilon _{3}\sigma ^{\prime }f_{4}\right)%
\end{array}%
\right) ,  \tag{3.11}
\end{equation}%
from the last equation we write%
\begin{equation}
\sigma ^{\prime }f_{3}+(\varepsilon _{2}\tau f_{2}+2f_{3}^{\prime })\sigma
=\varepsilon _{1}\kappa \Gamma _{2}+\varepsilon _{1}\sigma
^{2}f_{4}-\varepsilon _{3}f_{4}^{\prime \prime }.  \tag{3.12}
\end{equation}

Also, assuming the following equation, 
\begin{equation}
f_{4}(t)=\int \left( \varepsilon _{1}\varepsilon _{2}\int \kappa (t)\Gamma
_{2}dt+c_{2}\right) dt  \tag{3.13}
\end{equation}%
the differential equation (3.12) is obtain as%
\begin{equation*}
\sigma ^{\prime }f_{3}+(\varepsilon _{2}\tau f_{2}+2f_{3}^{\prime })\sigma
=\varepsilon _{1}\sigma ^{2}f_{4}
\end{equation*}%
and calculating the previous differential equations. Then, we get%
\begin{equation}
\sigma (t)=\frac{e^{\int \frac{(\varepsilon _{2}\tau f_{2}+2f_{3}^{\prime })%
}{f_{3}}dt}}{-\varepsilon _{1}\int \frac{f_{4}}{f_{3}}e^{-\int \frac{%
(\varepsilon _{2}\tau f_{2}+2f_{3}^{\prime })}{f_{3}}dt}dt+c_{3}}. 
\tag{3.14}
\end{equation}
\end{proof}
\end{theorem}

\begin{theorem}
Let $\frac{\partial \Omega }{\partial t}=f_{1}\overrightarrow{T}+f_{2}%
\overrightarrow{N}+f_{3}\overrightarrow{B_{1}}+f_{4}\overrightarrow{B_{2}}$
be arc-length parameter curve in $G_{1}^{4}$. If the flow of the curve $%
\Omega $ is inextensible then following partial differential equations hold%
\begin{equation*}
\varepsilon _{1}f_{2}^{\prime \prime }-\kappa f_{1}^{\prime }+2\tau
f_{3}^{\prime }-\kappa ^{\prime }f_{1}+\tau ^{\prime }f_{3}+\varepsilon
_{2}\tau ^{2}f_{2}+\sigma \tau \varepsilon _{2}f_{4}=0,
\end{equation*}%
\begin{equation*}
\tau (t)=\int \left( \Gamma _{1}^{\prime }-\varepsilon _{3}\Gamma _{2}\sigma
\right) dt;\sigma (t)=-\frac{1}{\Gamma _{1}\varepsilon _{2}}\left( \Gamma
_{2}^{\prime }+\Gamma _{3}\right) ,
\end{equation*}%
\begin{equation*}
\Gamma _{1}=\frac{-\varepsilon _{1}\kappa \xi _{1}}{\tau -\varepsilon _{1}}%
;\Gamma _{2}=\frac{1}{\sigma \varepsilon _{3}}\left( \Gamma _{1}^{\prime }-%
\frac{\partial \tau }{\partial t}\right) ;\Gamma _{3}=\Gamma _{1}\varepsilon
_{2}\sigma +\left( \frac{1}{\sigma \varepsilon _{3}}\left( \Gamma
_{1}^{\prime }-\frac{\partial \tau }{\partial t}\right) \right) ^{\prime },
\end{equation*}%
\begin{equation*}
\Gamma _{1}^{2}+\Gamma _{2}^{2}=2\int \frac{\partial \tau }{\partial t}%
\Gamma _{1}ds-\frac{\varepsilon _{3}}{\varepsilon _{2}}\int \Gamma
_{2}\Gamma _{3}ds,
\end{equation*}%
where $\xi _{1}=\kappa f_{1}+\varepsilon _{1}f_{2}^{\prime }-\tau f_{3}$ .

\begin{proof}
Noting that $\frac{\partial }{\partial s}\left( \frac{\partial N}{\partial t}%
\right) =\frac{\partial }{\partial t}\left( \frac{\partial N}{\partial s}%
\right) $ by making necessary arrangements, we get 
\begin{equation*}
\frac{\partial }{\partial s}\left( \frac{\partial N}{\partial t}\right) =%
\frac{\partial }{\partial s}\left( -\xi _{1}\overrightarrow{T}+\varepsilon
_{2}\Gamma _{1}\overrightarrow{B_{1}}+\Gamma _{2}\varepsilon _{3}%
\overrightarrow{B_{2}}\right)
\end{equation*}%
\begin{equation}
=-\xi _{1}^{\prime }\overrightarrow{T}+\varepsilon _{1}\left( 
\begin{array}{c}
-\xi _{1}\kappa \\ 
-\varepsilon _{2}\tau \Gamma _{1}%
\end{array}%
\right) \overrightarrow{N}+\varepsilon _{2}\left( 
\begin{array}{c}
\Gamma _{1}^{\prime } \\ 
-\varepsilon _{3}\Gamma _{2}\sigma%
\end{array}%
\right) \overrightarrow{B_{1}}+\varepsilon _{3}\left( 
\begin{array}{c}
\Gamma _{1}\varepsilon _{2}\sigma \\ 
+\Gamma _{2}^{\prime }%
\end{array}%
\right) \overrightarrow{B_{2}},  \tag{3.15}
\end{equation}%
where $-\xi _{1}=$ $-\kappa f_{1}-\varepsilon _{1}f_{2}^{\prime }+\tau f_{3}$
and by using $\frac{\partial \overrightarrow{B_{1}}}{\partial t},$ we write 
\begin{equation*}
\frac{\partial }{\partial t}\left( \frac{\partial N}{\partial s}\right) =%
\frac{\partial }{\partial t}\left( \varepsilon _{2}\tau \overrightarrow{B_{1}%
}\right) =\varepsilon _{2}\frac{\partial \tau }{\partial t}\overrightarrow{%
B_{1}}+\varepsilon _{2}\tau \frac{\partial \overrightarrow{B_{1}}}{\partial t%
}
\end{equation*}%
\begin{equation}
=\varepsilon _{2}\frac{\partial \tau }{\partial t}\overrightarrow{B_{1}}%
+\varepsilon _{2}\tau \left( -\tau f_{2}-\varepsilon _{2}f_{3}^{\prime
}+\sigma f_{4}\right) \overrightarrow{T}+\left( -\varepsilon _{1}\Gamma
_{1}\right) \overrightarrow{N}+\varepsilon _{3}\Gamma _{3}\overrightarrow{%
B_{2}},  \tag{3.16}
\end{equation}%
from equations (3.15) and (3.16), we can write%
\begin{eqnarray*}
-\xi _{1}^{\prime } &=&\varepsilon _{2}\tau (-\tau f_{2}-\varepsilon
_{2}f_{3}^{\prime }+\sigma f_{4}) \\
-\varepsilon _{1}\Gamma _{1} &=&\xi _{1} \varepsilon _{1}\kappa -\varepsilon
_{2}\varepsilon _{1}\tau \Gamma _{1} \\
\Gamma _{1}^{\prime }\varepsilon _{2}-\varepsilon _{3}\varepsilon _{2}\Gamma
_{2}\sigma &=&\varepsilon _{2}\frac{\partial \tau }{\partial t} \\
\Gamma _{1}\varepsilon _{3}\varepsilon _{2}\sigma +\Gamma _{2}^{\prime
}\varepsilon _{3} &=&\varepsilon _{3}\Gamma _{3}.
\end{eqnarray*}

When the third and fourth equations in the above equation system are
considered together and the necessary calculations are made, the following
equation is obtained.%
\begin{equation}
\Gamma _{1}^{2}+\Gamma _{2}^{2}=2\int \frac{\partial \tau }{\partial t}%
\Gamma _{1}ds-\frac{\varepsilon _{3}}{\varepsilon _{2}}\int \Gamma
_{2}\Gamma _{3}ds.  \tag{3.17}
\end{equation}

Also, from the third and fourth equation, we get%
\begin{equation}
\tau (t)=\int \left( \Gamma _{1}^{\prime }-\varepsilon _{3}\Gamma _{2}\sigma
\right) dt,  \tag{3.18}
\end{equation}%
\begin{equation}
\sigma (t)=-\frac{1}{\Gamma _{1}\varepsilon _{2}}\left( \Gamma _{2}^{\prime
}+\Gamma _{3}\right)  \tag{3.19}
\end{equation}%
and from the first equation, and from $-\xi _{1}=\kappa f_{1}-\varepsilon
_{1}f_{2}^{\prime }+\tau f_{3}$ we get%
\begin{equation}
\varepsilon _{1}f_{2}^{\prime \prime }-\kappa f_{1}^{\prime }+2\tau
f_{3}^{\prime }-\kappa ^{\prime }f_{1}+\tau ^{\prime }f_{3}+\varepsilon
_{2}\tau ^{2}f_{2}+\sigma \tau \varepsilon _{2}f_{4}=0.  \tag{3.20}
\end{equation}

Also, from the second, the third and the fourth equations we obtain
following equations%
\begin{equation}
\Gamma _{1}=\frac{-\varepsilon _{1}\kappa \xi _{1}}{\tau -\varepsilon _{1}}%
;\Gamma _{2}=\frac{1}{\sigma \varepsilon _{3}}\left( \Gamma _{1}^{\prime }-%
\frac{\partial \tau }{\partial t}\right) ;\Gamma _{3}=\Gamma _{1}\varepsilon
_{2}\sigma +\left( \frac{1}{\sigma \varepsilon _{3}}\left( \Gamma
_{1}^{\prime }-\frac{\partial \tau }{\partial t}\right) \right) ^{\prime }. 
\tag{3.21}
\end{equation}
\end{proof}
\end{theorem}

\begin{theorem}
Let $\frac{\partial \Omega }{\partial t}=f_{1}\overrightarrow{T}+f_{2}%
\overrightarrow{N}+f_{3}\overrightarrow{B_{1}}+f_{4}\overrightarrow{B_{2}}$
be arc-length parameter curve in $G_{1}^{4}$. If the flow of the curve $%
\Omega $ is inextensible then following partial differential equations hold%
\begin{eqnarray*}
\theta &=&e^{\varepsilon _{3}\int \sigma dt}\left( -\varepsilon _{2}\int
\tau \varphi e^{\varepsilon _{3}\int \sigma dt}dt+c_{3}\right) , \\
\tau (t) &=&-\int \left( \Gamma _{1}^{\prime }+\kappa \theta +\varepsilon
_{3}\sigma \Gamma _{1}\right) dt;\sigma (t)=c_{4}e^{-\varepsilon _{2}\int
\Gamma _{3}dt}, \\
\frac{\partial \sigma }{\partial t} &=&\frac{1}{\varepsilon _{3}}\left(
-\varepsilon _{3}\varepsilon _{2}\sigma \Gamma _{3}+\left( 1-\varepsilon
_{1}\varepsilon _{2}\right) \tau \Gamma _{1}\right) ,
\end{eqnarray*}%
where $\theta =-\tau f_{2}-\varepsilon _{2}f_{3}^{\prime }+\sigma
f_{4},\varphi =-\kappa f_{1}-\varepsilon _{1}f_{2}^{\prime }+\tau
f_{3},c_{3},c_{4}\in 
\mathbb{R}
.$
\end{theorem}

\begin{proof}
From the equality $\frac{\partial }{\partial s}\left( \frac{\partial B_{1}}{%
\partial t}\right) =\frac{\partial }{\partial t}\left( \frac{\partial B_{1}}{%
\partial s}\right) $ by making necessary arrangements, we get 
\begin{equation*}
\frac{\partial }{\partial s}\left( \frac{\partial B_{1}}{\partial t}\right) =%
\frac{\partial }{\partial s}\left( \theta \overrightarrow{T}+\left(
-\varepsilon _{1}\Gamma _{1}\right) \overrightarrow{N}+\Gamma
_{3}\varepsilon _{3}\overrightarrow{B_{2}}\right)
\end{equation*}%
\begin{equation}
=\theta ^{\prime }\overrightarrow{T}+\left( \Gamma _{1}^{\prime }\varepsilon
_{1}+\varepsilon _{1}\kappa \theta \right) \overrightarrow{N}+\left( -\Gamma
_{1}\varepsilon _{2}\varepsilon _{1}\tau -\varepsilon _{3}\varepsilon
_{2}\Gamma _{3}\sigma \right) \overrightarrow{B_{1}}+\left( \varepsilon
_{3}\Gamma _{3}^{\prime }\right) \overrightarrow{B_{2}},  \tag{3.22}
\end{equation}%
where $\theta =-\tau f_{2}-\varepsilon _{2}f_{3}^{\prime }+\sigma f_{4}.$%
\begin{eqnarray*}
\frac{\partial }{\partial t}\left( \frac{\partial B_{1}}{\partial s}\right)
&=&\frac{\partial }{\partial t}\left( -\varepsilon _{1}\tau \overrightarrow{N%
}+\varepsilon _{3}\sigma \overrightarrow{B_{2}}\right) \\
&=&-\varepsilon _{1}\frac{\partial \tau }{\partial t}\overrightarrow{N}%
-\varepsilon _{2}\tau \frac{\partial \overrightarrow{N}}{\partial t}%
+\varepsilon _{3}\frac{\partial \sigma }{\partial t}\overrightarrow{B_{1}}%
+\varepsilon _{3}\sigma \frac{\partial \overrightarrow{B_{1}}}{\partial t}
\end{eqnarray*}%
\begin{equation}
=\left( 
\begin{array}{c}
\varepsilon _{3}\sigma \theta \\ 
-\varepsilon _{2}\tau \varphi%
\end{array}%
\right) \overrightarrow{T}-\left( 
\begin{array}{c}
\frac{\partial \tau }{\partial t} \\ 
+\varepsilon _{3}\Gamma _{1}\sigma%
\end{array}%
\right) \overrightarrow{N}+\left( 
\begin{array}{c}
\varepsilon _{3}\frac{\partial \sigma }{\partial t} \\ 
-\tau \Gamma _{1}%
\end{array}%
\right) \overrightarrow{B_{1}}+\left( 
\begin{array}{c}
\varepsilon _{3}\sigma \Gamma _{3} \\ 
-\varepsilon _{2}\tau \Gamma _{2}%
\end{array}%
\right) \overrightarrow{B_{2}},  \tag{3.23}
\end{equation}%
where $\varphi =-\kappa f_{1}-\varepsilon _{1}f_{2}^{\prime }+\tau f_{3}.$
Applying the compatibility equations (3.22) and (3.23), we get%
\begin{eqnarray*}
\theta ^{\prime } &=&\varepsilon _{3}\sigma \theta -\varepsilon _{2}\tau
\varphi \\
\Gamma _{1}^{\prime }\varepsilon _{1}+\varepsilon _{1}\kappa \theta
&=&-\varepsilon _{1}\frac{\partial \tau }{\partial t}-\varepsilon
_{3}\varepsilon _{1}\Gamma _{1}\sigma \\
-\Gamma _{1}\varepsilon _{2}\varepsilon _{1}\tau -\varepsilon
_{3}\varepsilon _{2}\Gamma _{3}\sigma &=&\varepsilon _{3}\frac{\partial
\sigma }{\partial t}-\varepsilon _{2}^{2}\tau \Gamma _{1} \\
\varepsilon _{3}\Gamma _{3}^{\prime } &=&\varepsilon _{3}^{2}\sigma \Gamma
_{3}-\varepsilon _{3}\varepsilon _{2}\tau \Gamma _{2}
\end{eqnarray*}%
and by solving the first differential equation given above, the following
equation is obtained.%
\begin{equation}
\theta =e^{\varepsilon _{3}\int \sigma dt}\left( -\varepsilon _{2}\int \tau
\varphi e^{\varepsilon _{3}\int \sigma dt}dt+c_{3}\right) ;c_{3}\in 
\mathbb{R}
.  \tag{3.24}
\end{equation}

From the second differential equation, the following equality is obtained.%
\begin{equation}
\tau (t)=-\int \left( \Gamma _{1}^{\prime }+\kappa \theta +\varepsilon
_{3}\sigma \Gamma _{1}\right) dt.  \tag{3.25}
\end{equation}

From the third equation the following equality is obtained%
\begin{equation}
\frac{\partial \sigma }{\partial t}=\frac{1}{\varepsilon _{3}}\left(
-\varepsilon _{3}\varepsilon _{2}\sigma \Gamma _{3}+\left( 1-\varepsilon
_{1}\varepsilon _{2}\right) \tau \Gamma _{1}\right)  \tag{3.26}
\end{equation}%
and from $\varepsilon _{1},\varepsilon _{2}=\pm 1,$ we obtain%
\begin{equation}
\sigma (t)=c_{4}e^{-\varepsilon _{2}\int \Gamma _{3}dt}.  \tag{3.27}
\end{equation}
\end{proof}


\begin{theorem}
Let $\frac{\partial \Omega }{\partial t}=f_{1}\overrightarrow{T}+f_{2}%
\overrightarrow{N}+f_{3}\overrightarrow{B_{1}}+f_{4}\overrightarrow{B_{2}}$
be arc-length parameter curve in $G_{1}^{4}$. If the flow of the curve $%
\Omega $ is inextensible then following partial differential equations hold%
\begin{equation*}
-\varepsilon _{2}f_{4}^{\prime }+\sigma f_{3}=\varepsilon _{2}\int \sigma
\left( \tau f_{2}+\varepsilon _{2}f_{3}^{\prime }-\sigma f_{4}\right) ds,
\end{equation*}%
\begin{equation*}
\sigma (t)=\int \left( \Gamma _{3}^{\prime }+\varepsilon _{1}\tau \Gamma
_{2}\right) dt.
\end{equation*}

\begin{proof}
Observe that $\frac{\partial }{\partial s}\left( \frac{\partial B_{2}}{%
\partial t}\right) =\frac{\partial }{\partial t}\left( \frac{\partial B_{2}}{%
\partial s}\right) $ by making necessary arrangements, we get 
\begin{equation*}
\frac{\partial }{\partial s}\left( \frac{\partial B_{2}}{\partial t}\right)
=M^{\prime }\overrightarrow{T}+\left( M\varepsilon _{1}\kappa -\Gamma
_{2}^{\prime }\varepsilon _{1}+\varepsilon _{1}\varepsilon _{2}\tau \Gamma
_{3}\right) \overrightarrow{N}
\end{equation*}%
\begin{equation}
+\left( -\Gamma _{3}^{\prime }\varepsilon _{2}-\varepsilon _{1}\varepsilon
_{2}\tau \Gamma _{2}\right) \overrightarrow{B_{1}}+\left( -\varepsilon
_{2}\varepsilon _{3}\Gamma _{3}\sigma \right) \overrightarrow{B_{2}}, 
\tag{3.28}
\end{equation}%
where $M=-\varepsilon _{3}f_{4}^{\prime }+\sigma f_{3}.$%
\begin{equation}
\frac{\partial }{\partial t}\left( \frac{\partial B_{2}}{\partial s}\right)
=-\frac{\partial }{\partial t}\left( \varepsilon _{2}\sigma \overrightarrow{%
B_{1}}\right) =\varepsilon _{2}(-\sigma \theta \overrightarrow{T}%
+\varepsilon _{1}\Gamma _{1}\sigma \overrightarrow{N}-\frac{\partial \sigma 
}{\partial t}\overrightarrow{B_{1}}+\varepsilon _{3}\sigma \Gamma _{3}%
\overrightarrow{B_{2}}),  \tag{3.29}
\end{equation}%
where $\theta =-\tau f_{2}-\varepsilon _{2}f_{3}^{\prime }+\sigma f_{4}.$ If
we compare equations (3.28) and (3.29) we get%
\begin{eqnarray*}
M^{\prime } &=&-\varepsilon _{2}\sigma \theta \\
M\kappa -\Gamma _{2}^{\prime }+\varepsilon _{2}\tau \Gamma _{3}
&=&\varepsilon _{2}\Gamma _{1}\sigma \\
\Gamma _{3}^{\prime }+\varepsilon _{1}\tau \Gamma _{2} &=&\frac{\partial
\sigma }{\partial t} \\
\Gamma _{3}\sigma &=&0
\end{eqnarray*}%
and by solving the third differential equation we get the following equation.%
\begin{equation}
\sigma (t)=\int \left( \Gamma _{3}^{\prime }+\varepsilon _{1}\tau \Gamma
_{2}\right) dt.  \tag{3.30}
\end{equation}

From the first differential equation, the following equality is obtained.%
\begin{equation}
M(t)=-\varepsilon _{2}\int \sigma \theta ds\Rightarrow -\varepsilon
_{2}f_{4}^{\prime }+\sigma f_{3}=\varepsilon _{2}\int \sigma \left( \tau
f_{2}+\varepsilon _{2}f_{3}^{\prime }-\sigma f_{4}\right) ds.  \tag{3.31}
\end{equation}
\end{proof}
\end{theorem}

\section{The Energy and Pseudo-angle of vector fields in $G_{1}^{4}$}

In this section, the energy of vector fields depends on the constant since
the total bending energy can be characterized by the same variational
problem. In this context, the energy of the unit tangent vectors of the $s-$%
lines and $t-$lines on a particle moving in $G_{1}^{4}$ can be expressed.

\subsection{The energy of unit tangent vector of $s-$lines on a moving
particle in $G_{1}^{4}$}

Let $A$ be a moving particle in $\mathbf{G}_{1}^{4}$ such that it
corresponding to a curve $\Omega (s,u,t)$ with parameter $s$, which $s$ is
the distance along the $s-$lines of the curve in $s-$direction and tangent
vector of $s-$lines is described by $\frac{\partial \Omega ^{\prime }}{%
\partial s}$. Hence, by using Sasaki metric and the equations (2.10),
(2.11), (2.12) the energy, on the particle in vector field $\frac{\partial
\Omega ^{\prime }}{\partial s}$ can be written as
\begin{equation*}
energy_{\Omega _{s}^{\prime }}=E_{\Omega _{s}^{\prime }}=\frac{1}{2}\int
\rho _{s}(d\Omega ^{\prime }(\Omega ^{\prime }),d\Omega ^{\prime }(\Omega
^{\prime }))ds
\end{equation*}%
and by using the following equation 
\begin{equation*}
\rho _{s}(d\Omega ^{\prime }(\Omega ^{\prime }),d\Omega ^{\prime }(\Omega
^{\prime }))=\rho _{s}(T,T)+\rho _{s}(D_{T}^{T},D_{T}^{T})=1+\varepsilon
_{1}^{2}\kappa ^{2}\left\langle N,N\right\rangle =1+\varepsilon _{1}\kappa
^{2},
\end{equation*}%
we have 
\begin{equation}
E_{T_{s}}=\frac{s}{2}+\frac{\varepsilon _{1}}{2}\int \kappa ^{2}ds+c_{1}. 
\tag{4.1}
\end{equation}

Also, the energy on the particle in vector field $\frac{\partial N}{\partial
s}$ is written as 
\begin{equation*}
E_{N_{s}}=\frac{1}{2}\int \rho _{s}(dN(N),dN(N))ds
\end{equation*}%
and using the equation $\rho _{s}(dN(N),dN(N))=\varepsilon _{1}+\tau
^{2}\varepsilon _{2}$, we can write 
\begin{equation}
E_{N_{s}}=\frac{\varepsilon _{1}}{2}s+\frac{\varepsilon _{2}}{2}\int \tau
^{2}ds+c_{2}  \tag{4.2}
\end{equation}

Also, the energy on the particle in vector field $\frac{\partial B_{1}}{%
\partial s}$ is written as 
\begin{equation}
E_{B_{1s}}=\frac{1}{2}\int \rho _{s}(dB_{1}(B_{1}),dB_{1}(B_{1}))ds 
\tag{4.3}
\end{equation}%
and by using $\rho _{s}(dB_{1}(B_{1}),dB_{1}(B_{1}))=\varepsilon
_{2}+\varepsilon _{1}\tau ^{2}+\varepsilon _{3}\sigma ^{2},$ we get 
\begin{equation*}
E_{_{B_{1}s}}=\frac{\varepsilon _{2}}{2}s+\frac{\varepsilon _{1}}{2}\int
\tau ^{2}ds+\frac{\varepsilon _{1}}{2}\int \sigma ^{2}ds+c_{3}.
\end{equation*}

Similarly, the energy on the particle in vector field $\frac{\partial B_{2}}{%
\partial s}$ is obtained as 
\begin{equation}
E_{_{B_{2}s}}=\frac{\varepsilon _{3}}{2}\int (1+\sigma ^{2})ds=\frac{%
\varepsilon _{3}}{2}s+\frac{\varepsilon _{3}}{2}\int \sigma ^{2}ds+c_{4}. 
\tag{4.4}
\end{equation}

Then, we can give the following theorem.

\begin{theorem}
Energy of orthonormal frame field of $s-$lines with sasaki metric in $%
G_{1}^{4}$ are given as \bigskip 
\begin{eqnarray*}
E_{T_{s}} &=&\frac{s}{2}+\frac{\varepsilon _{1}}{2}\int \kappa ^{2}ds+c_{1}
\\
E_{N_{s}} &=&\frac{\varepsilon _{1}}{2}s+\frac{\varepsilon _{2}}{2}\int \tau
^{2}ds+c_{2} \\
E_{_{B_{1}s}} &=&\frac{\varepsilon _{2}}{2}s+\frac{\varepsilon _{1}}{2}\int
\tau ^{2}ds+\frac{\varepsilon _{1}}{2}\int \sigma ^{2}ds+c_{3} \\
E_{_{B_{2}s}} &=&\frac{\varepsilon _{3}}{2}s+\frac{\varepsilon _{3}}{2}\int
\sigma ^{2}ds+c_{4},
\end{eqnarray*}%
where $c_{i}\in IR.$
\end{theorem}

\subsection{The energy of unit tangent vector of $t-$lines on a moving
particle in $G_{1}^{4}$}

In this section, we calculate the energy of the unit tangent vector of $t-$%
lines of the inextensible curve $\Omega $ in $\mathbf{G}_{1}^{4}$ and we
also investigate the bending energy formula for inextensible curve given by
extended Serret-Frenet relations along the curve $\Omega (s,u,t)$ in $%
\mathbf{G}_{1}^{4}$. Let $A$ be a particle moving in $\mathbf{G}_{1}^{4}$
corresponding to a curve $\Omega (s,u,t)$ with parameter $t-$lines of the
curve in $t-$direction and the tangent vector of $t-$lines is described by $%
\frac{\partial T}{\partial t}$. Hence, by using Sasaki metric the energy on
the particle in vector field $\frac{\partial T}{\partial t}$ can be written
as
\begin{equation*}
energy_{T_{t}}=E_{T_{t}}=\frac{1}{2}\int \rho _{t}(dT(T),dT(T))dt
\end{equation*}%
\begin{equation*}
\rho _{t}(dT(T),dT(T))=1+\varepsilon _{1}\left( 
\begin{array}{c}
\varepsilon _{1}\kappa f_{1} \\ 
+f_{2}^{\prime } \\ 
-\varepsilon _{1}\tau f_{3}%
\end{array}%
\right) ^{2}+\varepsilon _{2}\left( 
\begin{array}{c}
\varepsilon _{2}\tau f_{2} \\ 
+f_{3}^{\prime } \\ 
-\varepsilon _{2}\sigma f_{4}%
\end{array}%
\right) ^{2}+\varepsilon _{3}\left( 
\begin{array}{c}
\varepsilon _{3}\sigma f_{3} \\ 
+f_{4}^{\prime }%
\end{array}%
\right) ^{2}
\end{equation*}%
then, we can write%
\begin{equation}
E_{T_{t}}=\frac{1}{2}\left\{ t+\varepsilon _{1}\int \left( 
\begin{array}{c}
(\varepsilon _{1}\kappa f_{1}+f_{2}^{\prime }-\varepsilon _{1}\tau f_{3})^{2}
\\ 
+\varepsilon _{2}(\varepsilon _{2}\tau f_{2}+f_{3}^{\prime }-\varepsilon
_{2}\sigma f_{4})^{2}+\varepsilon _{3}(\varepsilon _{3}\sigma
f_{3}+f_{4}^{\prime })^{2}%
\end{array}%
\right) dt\right\}   \tag{4.5}
\end{equation}%
the energy on the particle in vector field $\frac{\partial N}{\partial t}$
can be written as

\begin{equation*}
\rho _{t}(dN(N),dN(N))=\varepsilon _{1}+(-\kappa f_{1}+\tau
f_{3}-\varepsilon _{1}f_{2}^{\prime })^{2}+\varepsilon _{2}\Gamma
_{1}^{2}+\varepsilon _{3}\Gamma _{2}^{2}
\end{equation*}%
then, we can write%
\begin{equation}
E_{N_{t}}=\frac{1}{2}\left\{ \varepsilon _{1}t+\int ((\tau f_{3}-\kappa
f_{1}-\varepsilon _{1}f_{2}^{\prime })^{2}+\varepsilon _{2}\Gamma
_{1}^{2}+\varepsilon _{3}\Gamma _{2}^{2})dt\right\} ,  \tag{4.6}
\end{equation}%
the energy on the particle in vector field $\frac{\partial B_{1}}{\partial t}
$ can be written as

\begin{equation*}
\rho _{t}(dB_{1}(B_{1}),dB_{1}(B_{1}))=\varepsilon _{2}+(\sigma f_{4}-\tau
f_{2}-\varepsilon _{2}f_{3}^{\prime })^{2}+\varepsilon _{1}\Gamma
_{1}^{2}+\varepsilon _{3}\Gamma _{3}^{2}
\end{equation*}%
then, we can write%
\begin{equation}
E_{B_{1t}}=\frac{1}{2}\left\{ \varepsilon _{2}t+\int ((\sigma f_{4}-\tau
f_{2}-\varepsilon _{2}f_{3}^{\prime })^{2}+\varepsilon _{1}\Gamma
_{1}^{2}+\varepsilon _{3}\Gamma _{3}^{2})dt\right\} .  \tag{4.7}
\end{equation}%
the energy on the particle in vector field $\frac{\partial B_{2}}{\partial t}
$ can be written as

\begin{equation*}
\rho _{t}(dB_{2}(B_{2}),dB_{2}(B_{2}))=\varepsilon _{3}+(\sigma
f_{3}-\varepsilon _{3}f_{4}^{\prime })^{2}+\varepsilon _{1}\Gamma
_{2}^{2}+\varepsilon _{2}\Gamma _{3}^{2}
\end{equation*}%
then, we get%
\begin{equation}
E_{B_{2t}}=\frac{1}{2}\left\{ \varepsilon _{3}t+\int ((\sigma
f_{3}-\varepsilon _{3}f_{4}^{\prime })^{2}+\varepsilon _{1}\Gamma
_{2}^{2}+\varepsilon _{2}\Gamma _{3}^{2})dt\right\} .  \tag{4.8}
\end{equation}

Then, we can give the following theorem

\begin{theorem}
Energy of orthonormal frame field of $t-$lines with sasaki metric in $%
G_{1}^{4}$ are given as  
\begin{eqnarray*}
E_{T_{t}} &=&\frac{1}{2}\left\{ t+\varepsilon _{1}\int \left( 
\begin{array}{c}
(\varepsilon _{1}\kappa f_{1}+f_{2}^{\prime }-\varepsilon _{1}\tau f_{3})^{2}
\\ 
+\varepsilon _{2}(\varepsilon _{2}\tau f_{2}+f_{3}^{\prime }-\varepsilon
_{2}\sigma f_{4})^{2}+\varepsilon _{3}(\varepsilon _{3}\sigma
f_{3}+f_{4}^{\prime })^{2}%
\end{array}%
\right) dt\right\}  \\
E_{N_{t}} &=&\frac{1}{2}\left\{ \varepsilon _{1}t+\int ((\tau f_{3}-\kappa
f_{1}-\varepsilon _{1}f_{2}^{\prime })^{2}+\varepsilon _{2}\Gamma
_{1}^{2}+\varepsilon _{3}\Gamma _{2}^{2})dt\right\}  \\
E_{B_{1t}} &=&\frac{1}{2}\left\{ \varepsilon _{2}t+\int ((\sigma f_{4}-\tau
f_{2}-\varepsilon _{2}f_{3}^{\prime })^{2}+\varepsilon _{1}\Gamma
_{1}^{2}+\varepsilon _{3}\Gamma _{3}^{2})dt\right\}  \\
E_{B_{2t}} &=&\frac{1}{2}\left\{ \varepsilon _{3}t+\int ((\sigma
f_{3}-\varepsilon _{3}f_{4}^{\prime })^{2}+\varepsilon _{1}\Gamma
_{2}^{2}+\varepsilon _{2}\Gamma _{3}^{2})dt\right\} .
\end{eqnarray*}
\end{theorem}

\begin{theorem}
The pseudo angles of the Frenet-Serret vectors $\left\{
T,N,B_{1},B_{2}\right\} $ in $G_{1}^{4}$ along with $s-$lines coordinate
curve are expressed as 
\begin{eqnarray*}
A_{s}(T) &=&\frac{1}{2}\sqrt{\varepsilon _{1}}\int \kappa ds;\text{ }%
A_{s}(N)=\frac{1}{2}\sqrt{\varepsilon _{2}}\int \tau ds; \\
A_{s}(B_{1}) &=&\frac{1}{2}\int \left( \varepsilon _{1}\tau ^{2}+\varepsilon
_{3}\sigma ^{2}\right) ^{\frac{1}{2}}ds;\text{ }A_{s}(B_{2})=\frac{1}{2}%
\sqrt{\varepsilon _{2}}\int \sigma ds.
\end{eqnarray*}

\begin{proof}
By using the equation $A_{s}(T)=\frac{1}{2}\int \left\Vert \frac{\partial T}{%
\partial s}\right\Vert ds$ for the vector field $T$ and (2.9), the proof is
obvious. Similarly, the proof is clear for other vector fields $A_{s}(N)=%
\frac{1}{2}\int \left\Vert \frac{\partial N}{\partial s}\right\Vert ds,$ $%
A_{s}(B_{1})=\frac{1}{2}\int \left\Vert \frac{\partial B_{1}}{\partial s}%
\right\Vert ds,$ $A_{s}(B_{2})=\frac{1}{2}\int \left\Vert \frac{\partial
B_{2}}{\partial s}\right\Vert ds$ and by using (2.9).
\end{proof}
\end{theorem}

\begin{theorem}
The pseudo angles of the Frenet-Serret vectors $\left\{
T,N,B_{1},B_{2}\right\} $ in $G_{1}^{4}$ along with $t-$lines coordinate
curve are expressed as 
\begin{equation*}
A_{t}(T)=\frac{1}{2}\int \left( 
\begin{array}{c}
\varepsilon _{1}(\varepsilon _{1}\kappa f_{1}+f_{2}^{\prime }-\varepsilon
_{1}\tau f_{3})^{2} \\ 
+\varepsilon _{2}(\varepsilon _{2}\tau f_{2}+f_{3}^{\prime }-\varepsilon
_{2}\sigma f_{4})^{2}+\varepsilon _{3}(\varepsilon _{3}\sigma
f_{3}+f_{4}^{\prime })^{2}%
\end{array}%
\right) ^{\frac{1}{2}}dt
\end{equation*}%
\begin{equation*}
A_{t}(N)=\frac{1}{2}\int ((\tau f_{3}-\kappa f_{1}-\varepsilon
_{1}f_{2}^{\prime })^{2}+\varepsilon _{2}\Gamma _{1}^{2}+\varepsilon
_{3}\Gamma _{2}^{2})^{\frac{1}{2}}dt
\end{equation*}%
\begin{equation*}
A_{t}(B_{1})=\frac{1}{2}\int ((\sigma f_{4}-\tau f_{2}-\varepsilon
_{2}f_{3}^{\prime })^{2}+\varepsilon _{1}\Gamma _{1}^{2}+\varepsilon
_{3}\Gamma _{3}^{2})^{\frac{1}{2}}dt
\end{equation*}%
\begin{equation*}
A_{t}(B_{2})=\frac{1}{2}\int ((\sigma f_{3}-\varepsilon _{3}f_{4}^{\prime
})^{2}+\varepsilon _{1}\Gamma _{2}^{2}+\varepsilon _{2}\Gamma _{3}^{2})^{%
\frac{1}{2}}dt.
\end{equation*}
\end{theorem}

\begin{proof}
By using the equation $A_{t}(T)=\frac{1}{2}\int \left\Vert \frac{\partial T}{%
\partial t}\right\Vert dt$ for the vector field $T$ and (3.5), the proof is
obvious. Similarly, by using (3.5) the proof is clear for other vector
fields $A_{t}(N)=\frac{1}{2}\int \left\Vert \frac{\partial N}{\partial t}%
\right\Vert dt,$ $A_{t}(B_{1})=\frac{1}{2}\int \left\Vert \frac{\partial
B_{1}}{\partial t}\right\Vert dt,$ $A_{t}(B_{2})=\frac{1}{2}\int \left\Vert 
\frac{\partial B_{2}}{\partial t}\right\Vert dt$.
\end{proof}

\begin{example}
Let $\frac{\partial \Omega }{\partial t}$
be a curve with the Serret-Frenet vectors $\left\{ \overrightarrow{T},%
\overrightarrow{N},\overrightarrow{B_{1}},\overrightarrow{B_{2}}\right\} $
in $G_{1}^{4}$. If the flow of the curve $\Omega $ is an inextensible curve $\alpha (t)=(t,bt,a\cos kt,a\sin kt)$, then we can give the following graphs.

1) Extended Serret-Frenet relations variation graph to compare changes with respect to time for a given inextensible curve is drawn. For this
curve, we have 
\begin{equation*}
\kappa =ak^{2};\tau =k;\sigma =0;\Gamma _{1}=k\cos 2kt;\Gamma _{2},\Gamma
_{3}=0
\end{equation*}
\begin{eqnarray*}
\frac{\partial T}{\partial t} &=&(-ak^{2}t+b+ka\cos kt)\overrightarrow{N}%
+(kbt-ak\sin kt)\overrightarrow{B_{1}}+(ak\cos kt)\overrightarrow{B_{2}} \\
\frac{\partial N}{\partial t} &=&(-ak^{2}t+b+ka\cos kt)\overrightarrow{T}%
+k\cos 2kt\overrightarrow{B_{1}} \\
\frac{\partial B_{1}}{\partial t} &=&(-kbt+ak\sin kt)\overrightarrow{T}%
+k\cos 2kt\overrightarrow{N} \\
\frac{\partial B_{2}}{\partial t} &=&(-ak\cos kt)\overrightarrow{T}
\end{eqnarray*}%

2) Energy variation graph to compare changes of energies with respect to $t$ 
for a given inextensible curve $\alpha (t)$ is drawn%
$.$ For this curve, we have%
\begin{eqnarray*}
E_{T_{t}} &=&\frac{1}{2}\left\{ t-\int \left( (b-ak^{2}t+ka\cos
kt)^{2}+(kbt-ak\sin kt)^{2}+(ak\cos kt)^{2}\right) dt\right\}  \\
E_{N_{t}} &=&\frac{1}{2}\left\{ -t+\int ((b-ak^{2}t+ka\cos kt)^{2}+\left(
k\cos 2kt\right) ^{2})dt\right\}  \\
E_{B_{1t}} &=&\frac{1}{2}\left\{ t+\int ((-kbt+ak\sin kt)^{2}-\left( k\cos
2kt\right) ^{2})dt\right\}  \\
E_{B_{2t}} &=&\frac{1}{2}\left\{ t+\int (ak\cos kt)^{2}dt\right\}.
\end{eqnarray*}

3)Graph for the pseudo angles of the extended Frenet-Serret vector fields
along with $t-$lines coordinate to compare changes of pseudo angles with
respect to $t$ for given inextensible curve $\alpha (t)$ is drawn$.$ For this curve, we have  
\begin{equation*}
A_{t}(T)=\frac{1}{2}\int \left( (-ak^{2}t+b+ka\cos kt)^{2}+k^{2}(bt-a\sin
kt)^{2}+(ak\cos kt)^{2}\right) ^{\frac{1}{2}}dt
\end{equation*}%
\begin{equation*}
A_{t}(N)=\frac{1}{2}\int ((-ak^{2}t+b+ka\cos kt)^{2}+\left( k\cos 2kt\right)
^{2})^{\frac{1}{2}}dt
\end{equation*}%
\begin{equation*}
A_{t}(B_{1})=\frac{1}{2}\int ((-kbt+ak\sin kt)^{2}-\left( k\cos 2kt\right)
^{2})^{\frac{1}{2}}dt
\end{equation*}%
\begin{equation*}
A_{t}(B_{2})=\frac{1}{2}\int ak\cos ktdt.
\end{equation*}
\end{example}
\begin{figure}[!h]
\centering
\label{Fig1}
{\
\includegraphics[width=5cm,height=4cm]{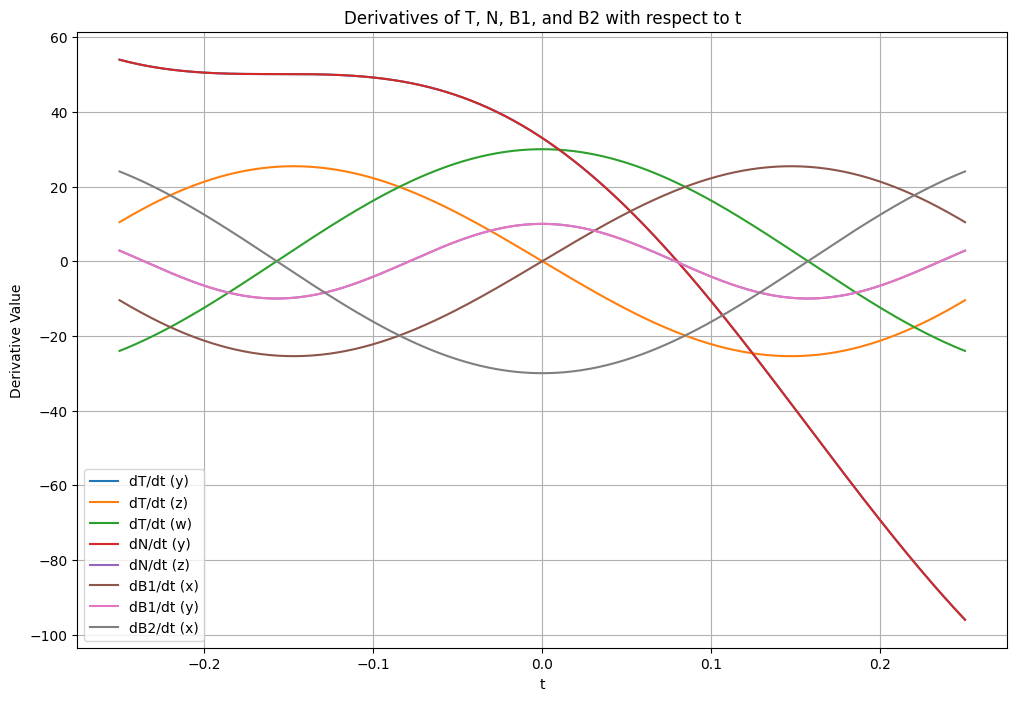}} \hspace*{.1cm}
{\
\includegraphics[width=5cm,height=4cm]{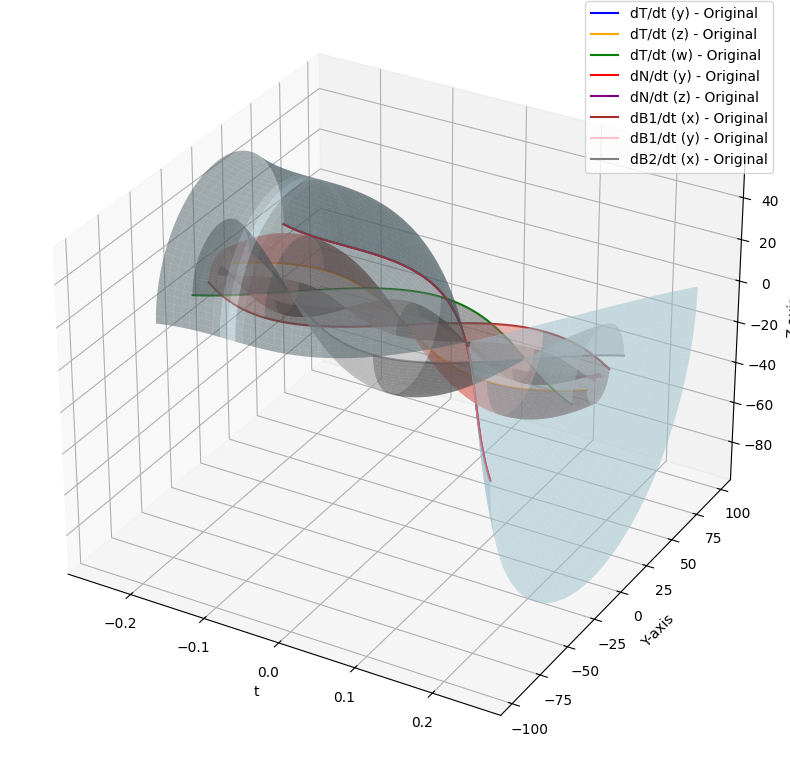}} \newline\caption{{Extended Serret-Frenet functions and the rotational surfaces generated by extended Serret-Frenet functions of $t-$lines of inextensible curve $\alpha .$} }%
\label{Fig1}%
\end{figure}


\begin{figure}[!h]
\centering
\label{Fig2}
{\
\includegraphics[width=5cm,height=4cm]{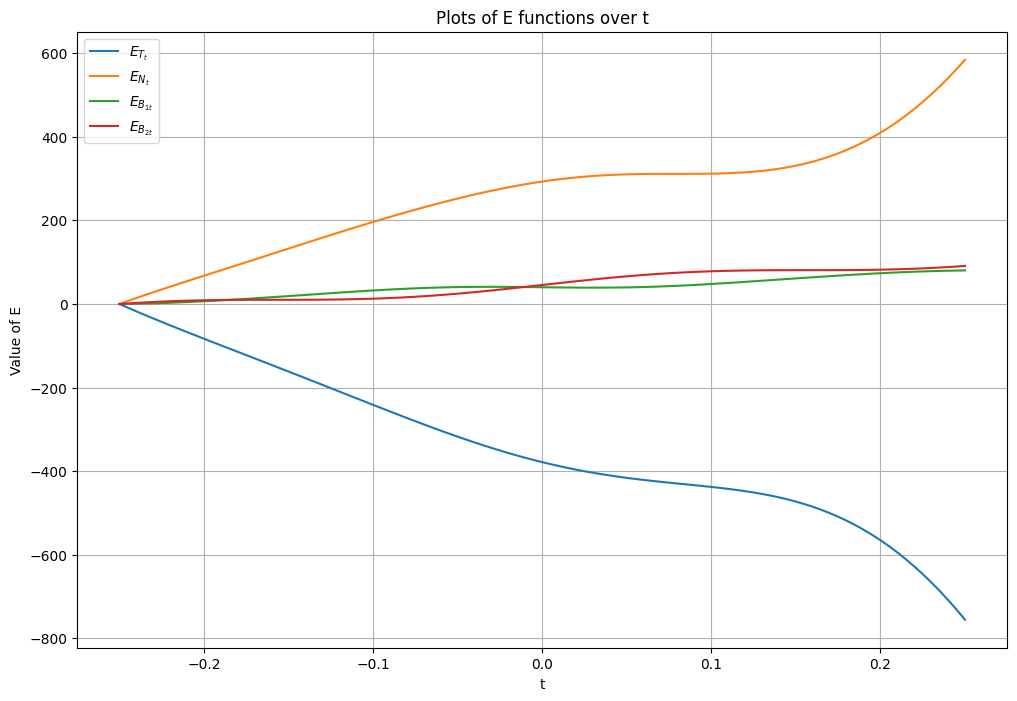} } \hspace*{.1cm}
{\
\includegraphics[width=6cm,height=5cm]{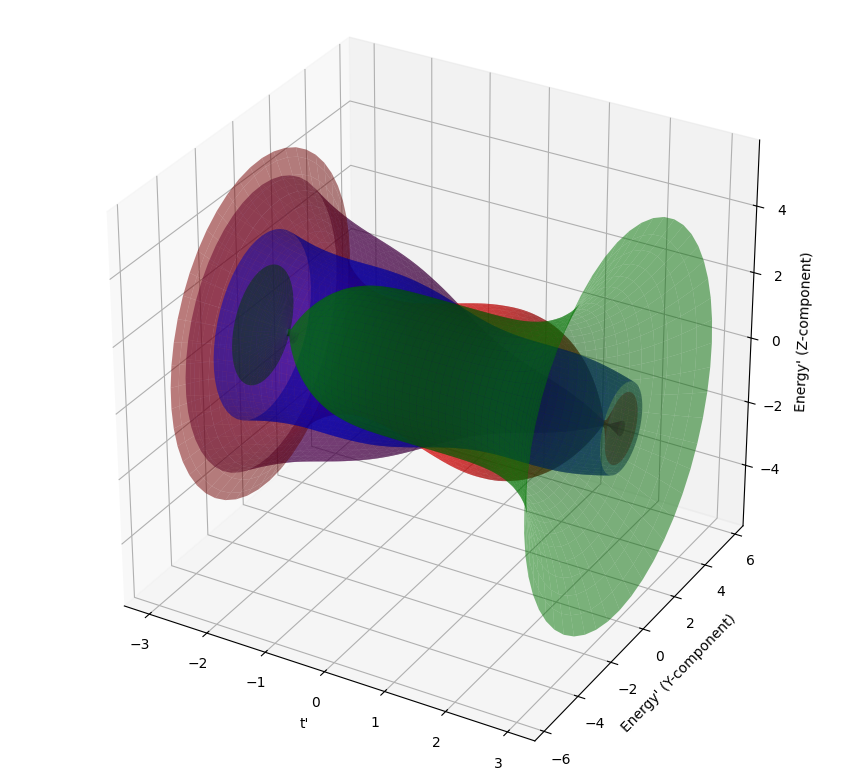}}
\newline\caption{{Energy functions and rotational surfaces generated by energy functions   of inextensible curve  $\alpha .$} }%
\label{Fig2}%
\end{figure}


\begin{figure}[!h]
\centering
\label{Fig3}
{\
\includegraphics[width=5cm,height=4cm]{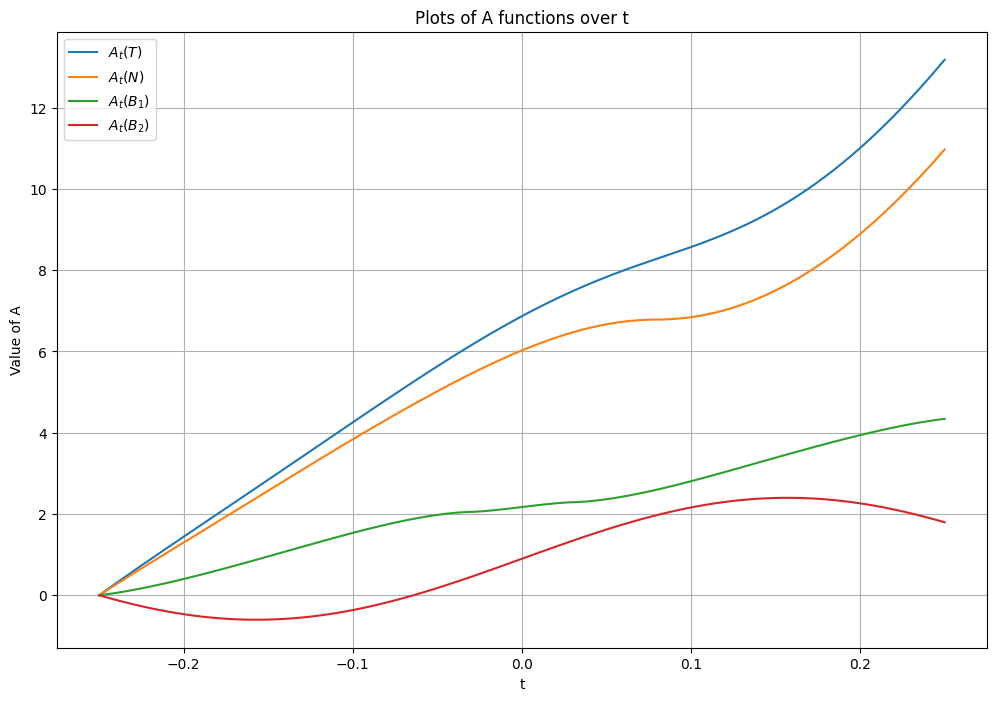}} \hspace*{.1cm}
{\
\includegraphics[width=5cm,height=4cm]{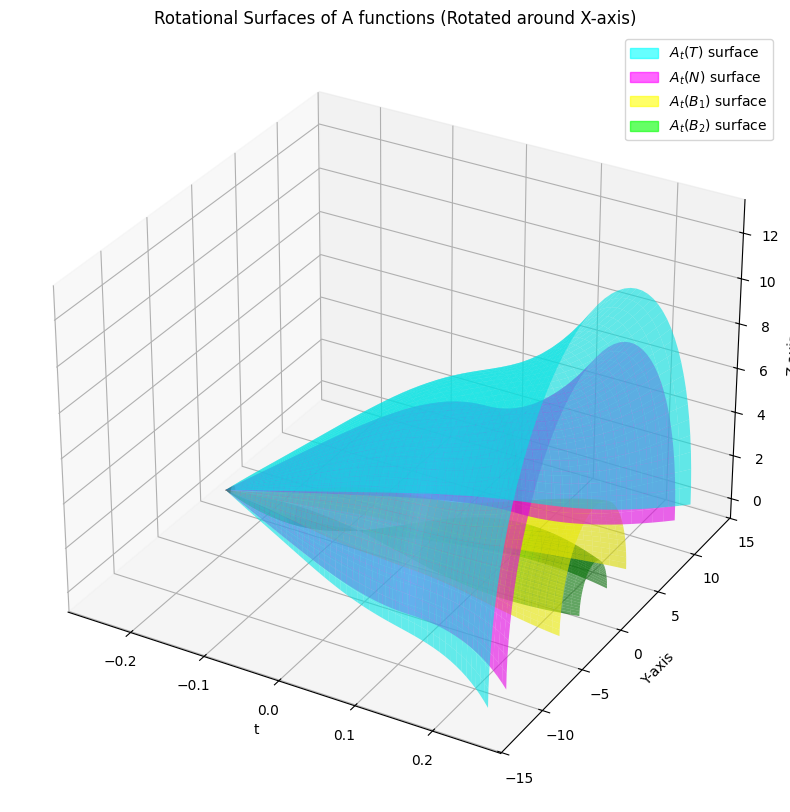}} \newline\caption{{Pseudo angle functions and rotational surfaces generated by pseudo angle functions of inextensible curve  $\alpha .$} }%
\label{Fig3}%
\end{figure}

In short, these surfaces are a mathematical representation of the dynamics of the frame in pseudo-Galilean 4-space and can be associated with energy or action in certain physical theories.

\subsection*{Acknowledgement}
The authors would like to thank referees for their valuable suggestions.

\end{document}